\documentclass[12pt]{article}
\usepackage{amssymb, amsmath,epsfig,amsthm}
\begin{document}
 \righthyphenmin=2

\def \<{{\langle}}
\def\>{{\rangle}}
\def\Z{{\Bbb Z}}
\def\N{{\mathbb N}}
\def\NN{{\mu_1}}
\def\T{{\cal T}}
\def \D{{\cal D}}
\def \C{{\cal C}}
\def \B{{\cal B}}
\def \Di{{\cal D}}
\def \di{{\Bbb D}}
\def \og{{\Bbb O}}
\def \H{{\cal H}}
\def \I{{\cal I}}
\def \Tht{{\Theta}}
\def \F{{\cal F}}
\def \V{{\cal V}}
\def \a{{\alpha}}
\def \l{{\lambda}}
\def \t{{\beta}}
\def \bt{{\beta}}
\def \dt{{\delta}}
\def \De{{\Delta}}
\def \ga{{\gamma}}
\def \kp{{\varkappa}}
\def \sgn{{\operatorname{sgn}}}
\def \R{{\mathbb R}}
\def \CR{{\cal R}}
\def \A {{\cal A}}
\def \a{{\alpha}}
\def \m{{\mu}}
\def \M{{\cal M}}
\def \g{{\gamma}}
\def \bt{{\beta}}
\def \la{{\lambda}}
\def \La{{\Lambda}}
\def \fee{{\varphi}}
\def \e{{\epsilon}}
\def \s{{\sigma}}
\def \Sm{{\Sigma}}
\def \sm{{\sigma_m}}
\def \td{\tilde}
\newcommand{\wh}{\widehat}
\newcommand{\wt}{\widetilde}
\newcommand{\ol}{\overline}
\def \supp{{\operatorname{supp}}}
\def \sign{{\operatorname{sign}}}
\def \sp{{\operatorname{Span}}}
\def \csp{{\overline{\operatorname{span}}}}
\def \Dff{{\operatorname{Diff}}}
\def \Dst{{\operatorname{Dist}}}
\def\charf\#1{\chi_{\lower3pt\hbox{$\scriptstyle \#1$}}} 
\catcode`\@=11
\def\NoLogo{\let\logo@\empty}
\catcode`\@=\active \NoLogo

\newtheorem{Theor}{Theorem}
\newtheorem{Lemm}{Lemma}
\newtheorem{Corollar}{Collary}
\theoremstyle{remark}
\newtheorem{Remar}{Remark}

\begin{center}
{\Large On Greedy Algorithms for dictionary with bounded
cumulative coherence.}\footnote{This research is partially
supported by Russian Foundation for Basic Research
project 08-01-00799 and 09-01-12173} \\
 Eugene Livshitz
\end{center}
\begin{abstract}
We discuss the upper and lower estimates for the rate of
convergence of Pure and Orthogonal Greedy Algorithms for
dictionary with bounded cumulative coherence.
\end{abstract}

\paragraph*{Introduction.}
Let $H$ be a real, separable Hilbert space equipped with an inner
product $\<\cdot,\cdot\>$ and the norm
$\|\cdot\|=\<\cdot,\cdot\>^{1/2}$. We say that a set $\D$,
$\D\subset H$ is a dictionary if
\begin{equation*}
 g\in\D\ \Rightarrow \|g\|=1,\mbox{ and }
\csp{\D}=H.
\end{equation*}
Recently the following problem has been intensively studied in
Approximation Theory and Numeral Analysis: to construct by element
$f\in H$ and $m\in\N$ an $m$-term combination
\begin{equation*}
f\rightarrow\sum_{k=1}^m c_k(f)g_k(f),\ c_k(f)\in\R,\ g_k(f)\in\D
\end{equation*}
that provides a good approximation for $f$. Greedy Algorithms turn
out to be effective for obtaining such $m$-term approximations
(see tutorial~\cite{T} for details). Two most popular of them are
defined below.

{\sc Pure Greedy Algorithm (PGA)} {\it Set  $f^{PGA}_0:=f\in H$,
$G_0^{PGA}(f,\Di) := 0$. For each $m\ge 0$ we inductively find
$g^{PGA}_{m+1}\in \D$ such that
\begin{equation}\label{introd-PGA-choice}
|\<f^{PGA}_m,g^{PGA}_{m+1}\>|= \sup_{g\in\D}|\<f^{PGA}_m,g\>|
\end{equation}
and define}
\begin{equation*}
 G_{m+1}^{PGA}(f,\D):= G_{m}^{PGA}(f,\Di)+\<f^{PGA}_m,g^{PGA}_{m+1}\>g^{PGA}_{m+1},
\end{equation*}
\begin{equation*}
f^{PGA}_{m+1}:= f-G_{m+1}^{PGA}(f,\Di) =
f^{PGA}_m-\<f^{PGA}_m,g^{PGA}_{m+1}\>g^{PGA}_{m+1}.
\end{equation*}

{\sc Orthogonal Greedy Algorithm (OGA)} {\it Set  $f^{OGA}_0:=f\in
H$, $G_0^{OGA}(f,\Di) := 0$. For each $m\ge 0$ we inductively find
$g^{OGA}_{m+1}\in \D$ such that
\begin{equation}\label{introd-PGA-choice}
|\<f^{OGA}_m,g^{OGA}_{m+1}\>|= \sup_{g\in\D}|\<f^{OGA}_m,g\>|
\end{equation}
and define}
\begin{equation*}
 G_{m+1}^{OGA}(f,\D):= Proj_{g^{OGA}_{1},\ldots,g^{OGA}_{m+1}}(f),
\end{equation*}
\begin{equation*}
f^{OGA}_{m+1}:= f-G_{m+1}^{OGA}(f,\Di).
\end{equation*}

Thus for $f\in H$  and each  $m\ge 1$ we construct $m$-term
approximations $G_{m}^{PGA}(f,\Di)$ and $G_{m}^{OGA}(f,\Di)$.

In this article we study the rate of convergence of Greedy
Algorithms for class $\A_0(\D)$  that is a set of finite linear
combination of elements from $\D$ and classes $\A^p(\D)$, $1\le p<
2$, defined below. For $M\ge 0$ we define
\begin{equation*}
\A^p(\D,M):=\overline{\{\sum_{\la\in\La}c_\la g^\la:
\sum_{\la\in\La}|c_\la|^p\le M^p,\ c_\la\in\R,\ g^\la\in\D,\
\sharp\La<\infty\}},
\end{equation*}
(where closure is taken in the norm of $H$). Set
\begin{equation*}
\A^p(\D):=\bigcup_{M\ge 0} \A^p(\D,M),
\end{equation*}
\begin{equation*}
|f|_p:=|f|_{\A^p(\D)}:=\inf\{M\ge 0: f\in \A^p(\D,M)\},\
f\in\A^p(\D).
\end{equation*}

From results of  R.A.~DeVore, V.N.~Temlyakov and E.D.~Livshitz
\cite{DT}, \cite{LT}, \cite{L3} it follows that Orthogonal Greedy
Algorithm does provide the optimal rate of convergence
$C|f|_1m^{-1/2}$ in $\A^1(\D)$, but  Pure Greedy Algorithm
doesn't. For narrower classes such as $\A_0(\D)$ the rate of
convergence of OGA could not be better than $Cm^{-1/2}$ and would
not be optimal. In the same time if dictionary $\D$ satisfies some
additional properties the rate of convergence of Greedy Algorithms
 (for some classes) could be essentially better. This area is
 called {\it Sparse Approximation} and has been intensively
 studied last time (\cite{GMS}, \cite{GN}, \cite{Tr},
\cite{DET}). In this article results will be formulated using the
notion of {\it cumulative coherence} of the dictionary introduced
by J.~Tropp~\cite{Tr}
\begin{equation}\label{coher-NN}
\NN(\D):=\sup_{g\in\D}\sum_{\wt g\in\D,\ \wt g\ne g} |\<\wt g,
g\>|.
\end{equation}
Above-mentioned articles contain the following basic results of
Sparse Approximation Theory.

\noindent{\bf Theorem A.} {\it Let $\D$ be a dictionary with
$\NN(\D)<1/2$ and $f\in\A_0(\D)$. Then
\begin{equation*}
G_m^{OGA}(f,\D)=f,\quad m \ge m_0,
\end{equation*}
\begin{equation*}
\|f-G_m^{PGA}(f,\D)\|=\|f_m\|\le C\exp(-c(f)m),\ m\ge 0.
\end{equation*}
}

For dictionaries with small  $\NN(\D)$ PGA provides optimal rate
of convergence in $\A^p(\D)$, $1\le p< 2$.

\begin{Theor}\label{ThA1Good} Let $\D$ be a dictionary with
$\NN(\D)<1/3$ and $f\in\A_1(\D)$. Then
\begin{equation*}
\|f-G_m^{PGA}(f,\D)\|=\|f_m\|\le |f|_1m^{-1/2},\ m\ge 0.
\end{equation*}
\end{Theor}

\begin{Theor}\label{ThA1GoodMod} Suppose $\D$ is a dictionary with
$\NN(\D)<1/3$, $p$, $1\le p<2$ and $f\in\A^p(\D)$. Then there
exist $C_1=C_1(p)>0$ and $C_2=C_2(\NN(\D))>0$ such that for any
$m\ge 1$
\begin{equation*}
\|f-G_m^{PGA}(f,\D)\|=\|f_m\|\le C_1C_2 |f|_p m^{-1/p+1/2}.
\end{equation*}
\end{Theor}

In the same time for big (but finite) values of
 $\NN(\D)$ Pure Greedy Algorithms can not always provide exponential rate of
 convergence,
 moreover one could be worse than $Cm^{-1/2}$:

\begin{Theor} \label{ThColow} There exists a dictionary
$\D$ with $\NN(\D)<\infty$,  $f_0\in\A_0(\D)$, $\bt >0$ and $C>0$
and that such for any $m \ge 1$ we have
\begin{equation*}
\|f_0-G_m^{PGA}(f_0,\D)\|=\|f_m\|\ge C m^{-1/2+\bt}.
\end{equation*}
\end{Theor}

\paragraph*{Properties of dictionaries with bounded cumulative coherence.}

It's easy to see that any dictionary with bounded cumulative
coherence in separable Hilbert space is countable. Suppose that
elements of dictionary are enumerated: $\D=\{g^\la\}_{\la \in\N}$.

\begin{Lemm}\label{Lm-A1-0}
Let $\D$ be a dictionary with $\NN(\D)<1/2$, $N>0$, $c_\nu\in\R$,
$g^\nu\in\D$, $1\le\nu\le\N$. Then the following inequalities
\begin{equation*}
(1-2\NN(\D))\sum_{\nu=1}^N c_\nu^2\le \left\|\sum_{\nu=1}^N c_\nu
g^\nu\right\|^2\le (1+2\NN(\D))\sum_{\nu=1}^N c_\nu^2
\end{equation*}
hold.
\end{Lemm}
\begin{proof}
Without loss of generality we can assume that
\begin{equation*}
|c_1|\ge |c_2|\ge\cdots\ge |c_N|.
\end{equation*}
We have
\begin{equation*}
\left\|\sum_{\nu=1}^N c_\nu g^\nu\right\|^2 = \<\sum_{\nu=1}^N
c_\nu g^\nu, \sum_{\nu=1}^N c_\nu g^\nu\>=\sum_{\nu=1}^N
\left(c_\nu^2 \<g^\nu,g^\nu\> + 2c_\nu\sum_{\eta = \nu+1}^n
c_\eta\<g^\nu,g^\eta\>\right).
\end{equation*}
Using (\ref{coher-NN}) and monotony of $|c_\nu|$ we estimate
\begin{multline*}
\left|c_\nu^2 \<g^\nu,g^\nu\> + 2c_\nu\sum_{\eta = \nu+1}^n
c_\eta\<g^\nu,g^\eta\> - c_\nu^2 \right|\le 2c_\nu\left|\sum_{\eta
=
\nu+1}^n c_\eta\<g^\nu,g^\eta\>\right|\le\\
\le 2c_\nu^2 \sum_{\eta = \nu+1}^n |\<g^\nu,g^\eta\>|\le
2c_\nu^2\NN(D).
\end{multline*}
Hence
\begin{equation*}
\left|\left\|\sum_{\nu=1}^N c_\nu g^\nu\right\|^2 - \sum_{\nu=1}^N
c_\nu^2\right| \le 2\NN(\D)\sum_{\nu=1}^N c_\nu^2.
\end{equation*}
\end{proof}

\begin{Lemm}\label{Lm-A1-1}
Suppose $\La\subset\N$ is a finite set of indexes and $\e>0$. If
for $f$ the representation
\begin{equation}\label{coh-A1-rep}
f=f_\e+ \sum_{\la\in\La}c_\la g^\la,\ c_\la\in\R,\ g^\la\in\D,\
\end{equation}
\begin{equation}\label{coh-A1-norm}
\sum_{\la\in\La}|c_\la|^p=|f|^p_p,
\end{equation}
\begin{equation}\label{coh-A1-eps}
\|f_\e\| <\e,
\end{equation}
holds, then for $\la_0\in\La$ we have
\begin{equation*}
\left|\<f,g^{\la_0}\>-c_{\la_0}\right|<\NN(\D)
\max_{\la\in\La}|c_\la|+\e,
\end{equation*}
and for $\la_0\not\in\La$ ---
\begin{equation*}
\left|\<f,g^{\la_0}\>\right|<\NN(\D) \max_{\la\in\La}|c_\la|+\e.
\end{equation*}
\end{Lemm}

\begin{proof}
Using representation (\ref{coh-A1-rep}) we write for $\la_0\in\La$
\begin{equation*}
\<f,g^{\la_0}\>-c_{\la_0}= \<f_\e+ \sum_{\la\in\La}c_\la
g^\la,g^{\la_0}\>-c_{\la_0}\<g^{\la_0},
g^{\la_0}\>=\sum_{\la\in\La,\ \la\ne\la_0}\<c_\la
g^\la,g^{\la_0}\>+\<f_\e,g^{\la_0}\>
\end{equation*}
and for $\la_0\not\in\La$
\begin{equation*}
\<f,g^{\la_0}\>=\<f_\e+ \sum_{\la\in\La}c_\la g^\la,g^{\la_0}\>=
\sum_{\la\in\La,\ \la\ne\la_0}\<c_\la
g^\la,g^{\la_0}\>+\<f_\e,g^{\la_0}\>.
\end{equation*}
To complete the proof we estimate using (\ref{coh-A1-eps}) and
Cauchy - Bunyakovsky - Schwarz inequality
\begin{multline*}
\left| \sum_{\la\in\La,\ \la\ne\la_0}\<c_\la
g^\la,g^{\la_0}\>+\<f_\e,g^{\la_0}\>\right|\le
\max_{\la\in\La}|c_\la|\sum_{\la\in\La,\ \la\ne\la_0}|\<
g^\la,g^{\la_0}\>|+\|f_\e\|\|g^{\la_0}\|<\\
 <
\max_{\la\in\La}|c_\la| \sum_{\wt g\in\D,\ \wt g\ne g^{\la_0}}
|\<\wt g, g^{\la_0}\>|+\e\le \NN(\D) \max_{\la\in\La}|c_\la|+\e.
\end{multline*}
\end{proof}

\begin{Lemm}\label{Lm-A1-2}
Let $\D$ be a dictionary with $\NN(\D)<1/3$, $f\in\A^p(\D)$ и
$m\ge 1$. Assume that for $n=m-1$, finite $\La\subset\N$ and
$\e>0$ the following representation
\begin{equation}\label{coh-A1-rep1}
f_n=f-G_n^{PGA}(f,\D)=f_\e+ \sum_{\la\in\La}c_{\la,n} g^\la,\
c_{\la,n}\in\R,\ g^\la\in\D,\  \|f_\e\| <\e
\end{equation}
holds. If
\begin{equation}\label{coh-A1-eps-new}
\e<\frac{1}{6}(1-3\NN(\D))\max_{\la\in\La}|c_{\la,m-1}|,
\end{equation}
then for $n=m$ we get (\ref{coh-A1-rep1}) with the same $\La$,
$f_\e$ and
\begin{equation*}
\sum_{\la\in\La}|c_{\la,m}|^p\le \sum_{\la\in\La}|c_{\la,m-1}|^p
-2^{-p}(1-3\NN(\D))^p\max_{\la_\in\La}|c_{\la,m-1}|^p,
\end{equation*}
\begin{equation}\label{coh-A1-maxi}
\max_{\la\in\La}|c_{\la,m}|\le \max_{\la\in\La}|c_{\la,m-1}|.
\end{equation}
\end{Lemm}

\begin{proof}
From the definition of PGA it follows, that for $m\ge 1$
\begin{equation*}
 f-G_m^{PGA}(f,\D)= f_m = f_{m-1}-G_1(f_{m-1}).
\end{equation*}
Therefore it's sufficient to prove the lemma for arbitrary
$f\in\A_p(\D)$ and $m=1$.

To reduce the notations we write $c_\la$ instead of $c_{\la,0}$,
$\la\in\La$. Taking into account (\ref{coh-A1-eps-new}) we have
\begin{equation}\label{coh-A1-eps1}
(1 - 2\NN(\D))\max_{\la\in\La}|c_\la| - 2\e>
(1-3\NN(\D))\max_{\la\in\La}|c_\la| - 3\e \ge \frac{1}{2}
(1-3\NN(\D))\max_{\la\in\La}|c_\la|> 0
\end{equation}
By Lemma~\ref{Lm-A1-1} we get
\begin{equation*}
\max_{\la\in\La}|\<f,g^\la\>|> \max_{\la\in\La} |c_\la| -
\NN(\D)\max_{\la\in\La} |c_\la|-\e= (1-\NN(\D))\max_{\la\in\La}
|c_\la| -\e,
\end{equation*}
for $\la\not\in\La$, using also (\ref{coh-A1-eps1}) we obtain
\begin{equation*}
|\<f,g^\la\>|< \NN(\D)\max_{\la\in\La} |c_\la|+\e<
(1-\NN(\D))\max_{\la\in\La} |c_\la| -\e.
\end{equation*}
Therefore there exists $\la_0\in\La$ such that
\begin{equation}\label{coh-A1-sup}
|\<f,g^{\la_0}\>|=\sup_{g\in\D}|\<f,g\>|>(1-\NN(\D))\max_{\la\in\La}
|c_\la| -\e.
\end{equation}
Using Lemma~\ref{Lm-A1-1}, we have
\begin{equation*}
|\<f,g^{\la_0}\>|<|c_{\la_0}|+\NN(\D)\max_{\la\in\La} |c_\la|+\e.
\end{equation*}
Combining last two inequalities, we obtain
\begin{equation*}
(1-\NN(\D))\max_{\la\in\La} |c_\la| -\e <
|c_{\la_0}|+\NN(\D)\max_{\la\in\La} |c_\la|+\e.
\end{equation*}
\begin{equation*}
|c_{\la_0}|>(1-2\NN(\D))\max_{\la\in\La} |c_\la|-2\e.
\end{equation*}
Without loss of generality we can assume that $c_{\la_0}\ge 0$,
that is
\begin{equation}\label{coh-cla0}
c_{\la_0}>(1-2\NN(\D))\max_{\la\in\La} |c_\la|-2\e.
\end{equation}
Applying Lemma~\ref{Lm-A1-1}, (\ref{coh-cla0}) and
(\ref{coh-A1-eps1}), we obtain
\begin{equation}\label{coh-A1-star1}
\<f,g^{\la_0}\> > c_{\la_0}-\NN(\D)\max_{\la\in\La} |c_\la|-\e
>(1-3\NN(\D))\max_{\la\in\La} |c_\la|-3\e\ge \frac{1}{2} (1-3\NN(\D))\max_{\la\in\La} |c_\la|,
\end{equation}
\begin{equation*}
\<f,g^{\la_0}\> < c_{\la_0}+\NN(\D)\max_{\la\in\La} |c_\la|+\e.
\end{equation*}
Hence by (\ref{coh-cla0}) and (\ref{coh-A1-eps1})
\begin{multline}\label{coh-A1-star2}
c_{\la_0} - \<f,g^{\la_0}\>\ge
c_{\la_0}-\left(c_{\la_0}+\NN(\D)\max_{\la\in\La}
|c_\la|+\e\right) \ge -\left(\NN(\D)\max_{\la\in\La}
|c_\la|+\e\right)\ge\\
\ge -c_{\la_0}+(1-3\NN(\D))\max_{\la\in\La} |c_\la|-3\e\ge
-c_{\la_0} + \frac{1}{2} (1-3\NN(\D))\max_{\la\in\La} |c_\la|.
\end{multline}
Combining (\ref{coh-A1-star1}) and (\ref{coh-A1-star2}), we
estimate
\begin{equation*}\label{coh-A1-accurate}
|c_{\la_0} - \<f,g^{\la_0}\>| + \frac{1}{2}
(1-3\NN(\D))\max_{\la\in\La} |c_\la|\le c_{\la_0},
\end{equation*}
\begin{equation}\label{coh-A1-accurate}
|c_{\la_0} - \<f,g^{\la_0}\>|^p\le c_{\la_0}^p - \left(\frac{1}{2}
(1-3\NN(\D))\max_{\la\in\La} |c_\la|\right)^{1/p}.
\end{equation}
If we set
\begin{equation*}
c_{\la,1}=c_\la=c_{\la,0},\ \la\in\La\setminus\{\la_0\},
\end{equation*}
\begin{equation*}
c_{\la_0,1}=c_{\la_0,0}-\<f,g^{\la_0}\>,
\end{equation*}
then statement of the lemma will folow from
(\ref{coh-A1-accurate}).
\end{proof}

\paragraph{Proof of Theorem~\ref{ThA1Good}.}
Lemma~\ref{Lm-A1-2} implies that for any $m\ge 0$
\begin{equation*}
|f_m|_1\le |f|_1.
\end{equation*}
 Using Lemma~3.5 from~\cite{DT}  and Lemma~\ref{Lm-A1-2} we have for $m\ge 0$
\begin{equation*}
|\<f_m,g_{m+1}\>|= \sup_{g\in\D}|\<f_m,g\>|\ge
\frac{\|f_m\|^2}{|f_m|_1}\ge \frac{\|f_m\|^2}{|f|_1}.
\end{equation*}
By definition of PGA
\begin{equation*}
\|f_{m+1}\|^2 =\|f_m\|^2 -\<f_m,g_{m+1}\>^2\le
\|f_m\|^2-\left(\frac{\|f_m\|^2}{|f|_1}\right)^2=\|f_m\|^2\left(1-\frac{\|f_m\|^2}{|f|^2_1}\right).
\end{equation*}
Applying Lemma~3.4 from~\cite{DT} for $a_m=\|f_{m-1}\|^2$ and
$A=|f|_1^2$ and taking into account the inequality
\begin{equation*}
a_1=\|f_0\|^2\le |f|_1^2,
\end{equation*}
we obtain that for $\{a_m\}_{m=1}^\infty$ such that
\begin{equation*}
a_{m+1}\le a_m\left(1-\frac{a_m}{|f|_1^2}\right),\  a_1\le A,
\end{equation*}
the following inequality
\begin{equation*}
a_{m}\le Am^{-1}
\end{equation*}
holds. Thus for $m\ge 1$ we have
\begin{equation*}
\|f_m\|=a_{m+1}^{1/2}\le |f|_1(m+1)^{-1/2}\le |f|_1 m^{-1/2}.
\end{equation*}
This completes the proof of the theorem.

\paragraph{Proof of Theorem~\ref{ThA1GoodMod}.}
Let $k\ge 1$ and $f\in\A_p$. For arbitrary $\e$
\begin{equation}\label{coh-A1-eps2}
0<\e<\frac{1}{6}(1-3\NN(\D))k^{-1/p}|f|_p,
\end{equation}
there exists representation (\ref{coh-A1-rep}) such that
inequalities (\ref{coh-A1-norm}) and (\ref{coh-A1-eps}) hold. We
claim that there exists $n$, $0\le n\le k$ such that
\begin{equation}\label{coh-A1-n1}
\max_{\la\in\La}|c_{\la,n}|\le c_1(p)c_2(\NN(\D)) k^{-1/p}|f|_p,
\end{equation}
\begin{equation}\label{coh-A1-n2}
\sum_{\la\in\La}|c_{\la,n}|^p\le\sum_{\la\in\La}|c_{\la,0}|^p=
|f|_p^p.
\end{equation}
 For every $m=1,\ldots,k$ for $n=m-1$ the representation
(\ref{coh-A1-rep1}) hold (beginning with $c_{\la,0}:=c_\la$) and
either
\begin{equation*}
\max_{\la\in\La}|c_{\la,m-1}|\le k^{-1/p}|f|_p,
\end{equation*}
in this case we can set $n=m-1$,
 or
\begin{equation*}
\max_{\la\in\La}|c_{\la,m-1}|\le k^{-1/p}|f|_p.
\end{equation*}
Then taking into account (\ref{coh-A1-eps2}) we have
(\ref{coh-A1-eps-new}). Therefore, by Lemma~\ref{Lm-A1-2} the
representation (\ref{coh-A1-rep1})  holds for $n=m$ and using
(\ref{coh-A1-norm}) and (\ref{coh-A1-maxi}) we have
\begin{multline*}
0\le\sum_{\la\in\La}|c_{\la,m}|^p\le
\sum_{\la\in\La}|c_{\la,m-1}|^p
-2^{-p}(1-3\NN(\D))^p\max_{\la\in\La}|c_{\la,m-1}|^p =\cdots=\\=
\sum_{\la\in\La}|c_{\la,0}|^p-\sum_{n=1}^m
2^{-p}(1-3\NN(\D))^p\max_{\la\in\La}|c_{\la,m-1}|^p\le
|f|_p^p-m2^{-p}(1-3\NN(\D))^p\max_{\la\in\La}|c_{\la,m}|^p.
\end{multline*}
\begin{equation*}
\max_{\la\in\La}|c_{\la,m}|\le 2(1-3\NN(\D))^{-1}
m^{-1/p}|f|_p,\quad \sum_{\la\in\La}|c_{\la,m}|^p\le|f|_p^p,
\end{equation*}
This provides (\ref{coh-A1-n1}) and (\ref{coh-A1-n2}) for $n=k$.

Using (\ref{coh-A1-n1}) and (\ref{coh-A1-n2}), we estimate
\begin{multline*}
\sum_{\la\in\La}|c_{\la,n}|^2 =
\sum_{\la\in\La}|c_{\la,n}|^p|c_{\la,n}|^{2-p}=\left(\max_{\la\in\La}|c_{\la,n}|\right)^{2-p}\sum_{\la\in\La}|c_{\la,n}|^p\le\\
\le \left(c_3(p)c_4(\NN(\D))
k^{-\frac{2-p}{p}}|f|^{2-p}_p\right)|f|_p^p=c_3(p)c_4(\NN(\D)
k^{-2/p + 1}|f|_p^2.
\end{multline*}
Applying Lemma~\ref{Lm-A1-0} and (\ref{coh-A1-rep1}), we obtain
\begin{multline*}
\|f_k\|\le\|f_n\| \le \|\sum_{\la\in\La}c_{\la,n} g^\la\|
+\|f_\e\|\le
\left((1+2\NN(\D))\sum_{\la\in\La} c_{\la,n}^2\right)^{1/2}+\e \le\\
\le C_1(p)C_2(\NN(\D))k^{-1/p+1/2}|f|_p +\e.
\end{multline*}
Since $\e>0$ can be arbitrary small the last inequality completes
the proof of Theorem~\ref{ThA1GoodMod}. $\square$

\end{document}